\documentclass[UTF8,12,a4paper]{amsart}

\usepackage{amsmath,amsthm,amsfonts,amssymb,amscd}
\usepackage{latexsym}
\usepackage{setspace}   
\usepackage{fancyhdr}
\usepackage{tikz-cd}    
\usepackage{enumerate}
\usepackage[all]{xy}
\usepackage[mathscr]{euscript}
\usepackage{color}
\usepackage{tikz}
\usepackage{tcolorbox}
\usepackage{enumerate}

\usepackage{geometry}  
\usepackage{graphicx}

\usepackage{cite}
\usepackage[colorlinks]{hyperref}
\definecolor{darkgreen}{rgb}{0,0.35,0}
\newtcolorbox{mybox}[2][]{width=10cm,colback = red!5!white, colframe = green!75!black, fonttitle = \bfseries,colbacktitle = red!55!yellow, enhanced,attach boxed title to top left={yshift=-2mm},	title=#2,#1}

\theoremstyle{plain} 
\newtheorem{thm}{{Theorem}\hspace{0.05pt}}[section]
\newtheorem{prop}[thm]{{Proposition}\hspace{0.05pt}}
\newtheorem{cor}[thm]{Corollary\hspace{0.05pt}}
\newtheorem{lem}[thm]{{Lemma}\hspace{0.05pt}}

\theoremstyle{definition}
\newtheorem{dfn}[thm]{{Definition}\hspace{0.05pt}}
\newtheorem{cons}[thm]{{Construction}\hspace{0.05pt}}
\newtheorem{dfn/prop}[thm]{{Definition/Proposition}\hspace{0.05pt}}
\newtheorem{rem}[thm]{{Remark}\hspace{0.05pt}}

\newtheorem*{ntt}{{Notation}\hspace{0.05pt}}
\newtheorem*{ack}{{Acknowledgments}\hspace{0.05pt}}

\newcommand{\K}{\mathrm{K}}

\newcommand{\A}{\mathscr{A}}
\newcommand{\B}{\mathscr{B}}
\newcommand{\C}{\mathscr{C}}
\newcommand{\D}{\mathscr{D}}
\newcommand{\E}{\mathscr{E}}

\newcommand{\ProFin}{\mathrm{ProFin}}

\newcommand{\An}{\mathrm{An}}
\newcommand{\Sp}{\mathrm{Sp}}

\newcommand{\Fun}{\mathrm{Fun}}

\newcommand{\Shv}{\mathrm{Shv}}
\newcommand{\op}{\mathrm{op}}

\newcommand{\Hom}{\mathrm{Hom}}

\newcommand{\fib}{\mathrm{fib}}

\newcommand{\cont}{\mathrm{cont}}
\newcommand{\dual}{\mathrm{dual}}

\newcommand{\CondAn}{\mathrm{CondAn}}

\newcommand{\CAlg}{\mathrm{CAlg}}

\newcommand{\light}{\mathrm{light}}
\newcommand{\Corr}{\mathrm{Corr}}

\newcommand{\Cat}{\mathrm{Cat}}

\begin{document}
	\small 
	
	\title{The universal continuous six functor formalism on light condensed anima}
	\date{\today}
	
	\author{Li He}
	\address{Graduate School of Mathematics, Nagoya University, 464-8602}
	\email{m21015y@math.nagoya-u.ac.jp}
	
	\maketitle
	\begin{abstract}
		After the universal property of the six functor formalism $\Shv(-;\Sp)$ on locally compact Hausdorff spaces given by Zhu,  we show that the six functor formalism $\Shv(-;\Sp)$
		on light condensed anima in the sense of Heyer-Mann is initial among all six functor formalisms $D$ satisfying 
		some mild conditions, and then we present some applications.
	\end{abstract}
	\tableofcontents
	
	\section{Introduction}
	
In mathematics, the six functor formalisms are fundamental to understand topology and geometry.
Recent years, people started to study the six functor formalisms from higher categorical perspective.
For example, Liu-Zheng and Mann (\cite[Appendix A.5]{Man22}) gave a very concise definition of six functor formalisms.

Heyer-Mann in \cite{HM24} says that a \textit{geometric setup}
is a pair $(\C,E)$, where $\C$ is a category has all finite limits, and $E$ is a family of morphisms in $\C$,
which is closed under composition, base change,
and contains all isomorphisms, and is closed under passing to diagonals.

Given a geometric setup $(\C,E)$,
there is a \textit{category of correspondences}
$\mathrm{Corr}(\C,E)$ associated with $(\C,E)$, which can be described as follows:
\begin{itemize}
	\item[(a)] objects in $\mathrm{Corr}(\C,E)$
	are objects in $\C$.
	\item [(b)] morphisms from $Y$ to $X$ are correspondences
	$Y \stackrel{f}{\leftarrow} 
	Z \stackrel{g}{\rightarrow} X $,
	where $g$ lies in $E$.
\end{itemize}
One can promote the category $\mathrm{Corr}(\C,E)$
to a symmetric monoidal category
$\mathrm{Corr}(\C,E)^\otimes$ whose underlying category is
$\mathrm{Corr}(\C,E)$. See \cite[Section 2.3]{HM24}.

Liu-Zheng and Mann says that a \textit{3-functor formalism}
is a lax symmetric monoidal functor
$$D:\mathrm{Corr}(\C,E)^\otimes\to {\Pr}^{L,\otimes},$$
where the symmetric monoidal structure on $\Pr^{L,\otimes}$ is given by the Lurie's tensor product.

Now, given a 3-functor formalism
$D:\mathrm{Corr}(\C,E)^\otimes\to {\Pr}^{L,\otimes},$
we are able to obtain the following functors:
\begin{itemize}
	\item [(1)] For any map $f:Y\to X$,
	the correspondence $Y \stackrel{f}{\leftarrow} X= X $
	gives the pullback functor
	$f^*:D(X)\to D(Y)$.
	\item [(2)] For any map $g:Y\to X$ in $E$,
	the correspondence $Y=Y \stackrel{g}{\rightarrow} X$
	gives the functor $g_!:D(Y)\to D(X)$.
	\item [(3)] The tensor product functor
	$-\otimes-:D(X)\otimes D(X)\to D(X)$
	is given by the composition
	$ D(X)\otimes D(X)\to D(X\times X)\to D(X)$,
	where the first functor is given by the lax symmetric monoidal functor $D$, and the second functor is $\Delta^*$,
	where $\Delta:X\to X\times X$ is the diagonal.
\end{itemize}

Then one can make the following definition:

\begin{dfn}
	A \textit{6-functor formalism}
	is a 3-functor formalism
	$D:\mathrm{Corr}(\C,E)^\otimes\to {\Pr}^{L,\otimes}$,
	such that the functors
	$-\otimes A$, $f^*$ and $f_!$ admit right adjoints.
\end{dfn}

A fundamental example of six functor formalisms is the six functor formalism on locally compact Hausdorff spaces.
Given a locally compact Hausdorff space $X$,
we can construnction
the category $\Shv(X;\Sp)$ of $\Sp$-valued sheaves on $X$.
And we know that the assignment $X\mapsto \Shv(X;\Sp)$
promotes to a six functor formalism
$$\Shv(-;\Sp):
\mathrm{Corr}(\mathrm{LCH},\mathrm{all})^\otimes\to {\Pr}^{L,\otimes},$$
where $\mathrm{LCH}$ is the category of locally compact Hausdorff spaces. See \cite[7.4]{Sch25}.

In his work \cite{Zhu25},
Qingchong Zhu showed that
the six functor formalism
$\Shv(-;\Sp):
\mathrm{Corr}(\mathrm{LCH})^\otimes\to {\Pr}^{L,\otimes}$
is initial among all \textit{continuous} six functor formalisms
(See \cite[4.18]{Zhu25}). 
Here, a six functor formalism
$D:\mathrm{Corr}(\mathrm{LCH})^\otimes\to {\Pr}^{L,\otimes}$
is called continuous, if it satisfies the following conditions:
\begin{itemize}
	\item [(1)] For any locally compact Hausdorff space $X$, the category $D(X)$ is dualizable.
	\item [(2)] $D$ satisfies canonical descent.
	That is, for any locally compact Hausdorff space $X$,
	the assignment 
	$\mathrm{Open}(X)^\op\to {\Pr}^L; U \mapsto D(U) $
	is a sheaf.
	\item [(3)] $D$ sasisfies profinite descent.
	That is, for a cofiltered limit $X\simeq \lim_i X_i$ of compact Hausdorff spaces $X_i$, the induced functor
	$D(X)\to \lim_i D(X_i) $
	is an equivalence.
	\item [(4)] For any hypercomplete locally compact Hausdorff space $X$, the family
	$\{x^*:D(X)\to D(*) \}_{x\in X}$
	is jointly conservative.
\end{itemize}

In other words, if $D$ is a six functor formalism on locally compact Hausdorff spaces satisfying conditions (1)-(4),
then there exists uniquely a morphism
$$\alpha:\Shv(-;\Sp)\to D(-)$$
 between six functor formalisms.

Given a six functor formlism $D$ on a geometric setup $(\C,E)$, Heyer-Mann gave an extension of $D$
so that we get a six functor formalism
on $(\mathrm{HypShv}(\C), E^\prime)$, where $E^\prime$ is a family in $\mathrm{HypShv}(\C)$. See \cite[Section 3.4]{HM24}.
In particular, one can extend the six functor formalism
$$\Shv(-;\Sp):\mathrm{Corr}(\ProFin^\mathrm{light})\to
{\Pr}^L$$
 on light profinite sets to the six functor formalism
$$\Shv(-;\Sp):\mathrm{Corr}(\CondAn^\mathrm{light},E)\to
{\Pr}^L$$
 on light condensed anima.

Now, we can state our main result:

	\begin{thm} \label{main result}
		The six functor formalism
		$$\Shv(-;\Sp):\mathrm{Corr}(\CondAn^\mathrm{light},E)\to
		{\Pr}^L_{\kappa}$$
		is initial among all six functor formalisms
		$D:\mathrm{Corr}(\CondAn^\mathrm{light},E)\to
		{\Pr}^L_{\kappa}$ satisfying the following conditions:
		\begin{itemize}
			\item [(a)]
			For any light profinite set $S$,
			the $\infty$-category $D(S)$ is dualizable,
			and the family $\{s^*:D(S)\to D(*)\}_{s\in S}$
			is jointly conservative.
			\item [(b)]
			The functor $D:\CondAn^{\mathrm{light},\op}\to {\Pr}^L_{\kappa}$
			preserves limits.
		\end{itemize}
		We denote the unique morphism
		by $\alpha:\Shv(-;\Sp)\to D(-)$.
	\end{thm}

Given a six functor formalism $D:\mathrm{Corr}(\C,E)\to {\Pr}^L$,	
Heyer-Mann introduced the concepts of $D$-suave maps 
and $D$-prim maps. See \cite[Section 4.4, Section 4.5]{HM24}.
Using these concepts, we are able to show the following result.

	\begin{thm} 
		Let $f:Y\to X$ be a map in $E$
		and $\alpha:D\simeq \Shv(-;\Sp) \to D^\prime$ the unique morphism from Theorem \ref{main result}. We have:
		\begin{itemize}
			\item [(1)]If $f:Y\to X$ is $D$-suave,
			then the commutative square
			\begin{center}
				\begin{tikzcd}
					D(Y)\ar[r,"\alpha_Y"]\ar[d,"f_!"]
					& D^\prime(Y)\ar[d,"f_!"]\\
					D(X)\ar[r,"\alpha_X"] & D^\prime(X)
				\end{tikzcd}
			\end{center}
			is vertically right adjointable. That is, the canonical map
			$ \alpha_Y f^! \to f^! \alpha_X$
			is an equivalence.
			\item [(2)]If $f:Y\to X$ is $D$-prim,
			then the commutative square
			\begin{center}
				\begin{tikzcd}
					D(X)\ar[r,"\alpha_X"]\ar[d,"f^*"]
					& D^\prime(X)\ar[d,"f^*"]\\
					D(Y)\ar[r,"\alpha_Y"] & D^\prime(Y)
				\end{tikzcd}
			\end{center}
			is vertically right adjointable. That is, the canonical map
			$ \alpha_X f_* \to  f_* \alpha_Y$
			is an equivalence.
		\end{itemize}
	\end{thm}

With the morphism $\alpha:\Shv(-;\Sp)\to D(-)$ in hand,
We can give the following application.
	
	\begin{thm}
		Let $D:  \mathrm{Corr}(\CondAn^\mathrm{light},E)\to
	{\Pr}^L_{\kappa}$ be a 6-functor formalism
	satisfying conditions (a) and (b) in Theorem \ref{main result}.
		For any light condensed anima $X$
		with projection $f_X:X\times \mathbb{R} \to X$,
		the pullback functor
		$$ f_X^*:D(X)\to D(X\times \mathbb{R}) $$
		is fully faithful.
	\end{thm}

	\begin{ack}
		The author would like to thank Lars Hesselholt 
		for some helpful suggestions.
	\end{ack}

	\section{Six functor formalisms on light profinite sets}
	In this section, we present the universal property of six functor formalisms on light profinite sets.
	
		We always fix a cardinal $\kappa>\omega$.
	And let $\Pr^L_\kappa$
	be the subcategory of $\Pr^L$
	spanned by $\kappa$-presentable $\infty$-categories
	and colimit-preserving functors which preserves $\kappa$-compact objects.

	We start with the definition of light profinite sets and
	light condensed anima from 
	\cite[Section 4]{Cla25}.

\begin{dfn}
	\begin{itemize}
		\item [(1)]A \textit{light profinite set} is a topological space $S$
		which can be written as $S=\varprojlim_i S_i$,
		the sequence inverse limit of finite discrete topological spaces $S_i$. The full subcategory of $\mathrm{Top}$ spanned by light profinite sets is denoted by $\mathrm{ProFin}^\light$.
		\item [(2)] A \textit{light condensed anima} 
		is a functor $X:\mathrm{ProFin}^{\light,\op}\to \An$ such that:
		\begin{itemize}
			\item [(a)]
			$X(\sqcup_{i\in I}S_i )\stackrel{\sim}{\to}
			\prod_{i\in I} X(S_i)$, where $I$ is a finite set and
			each $S_i$ is a light profinite set.
			\item [(b)] If $S_\bullet \to S$ is a hypercover of $S$
			in $\mathrm{ProFin}^\light$, then the canonical map
			$ X(S)\to \lim_{[n]\in \Delta} X(S_n) $
			is an equivalence.
		\end{itemize}
		The full subcategory of 
		$\Fun( \mathrm{ProFin}^{\light,\op}, \An )$
		spanned by light condensed anima is denoted by
		$\CondAn^\light$.
	\end{itemize}
	
\end{dfn}

	\begin{ntt}
		Given a geometric setup $(\C,E)$,
		we use $6\mathrm{FF}(\C,E)$ to denote
		the $\infty$-category of all six functor formalisms on $(\C,E)$. That is, we have
		$$6\mathrm{FF}(\C,E)\simeq
		\Fun^\mathrm{lax}( \mathrm{Corr}(\C,E)^\otimes,	{\Pr}^{L,\otimes}_{\kappa} )
		\simeq
		\CAlg( \Fun( \mathrm{Corr}(\C,E),
		{\Pr}^L_{\kappa})^\otimes ).
		$$
	\end{ntt}

	\begin{rem} \label{lax sym mon Corr}
		Recall from \cite[2.3.5]{HM24}
		that every morphism
		$(\C,\C_0)\to (\D,\D_0)$ of geometric setups
		will induce a lax symmetric monoidal functor
		$$\mathrm{Corr}(\C,\C_0)^\otimes
		\to \mathrm{Corr}(\D,\D_0)^\otimes .$$
		
\begin{itemize}
	\item [(1)]	Since $(\ProFin^\light,\text{all})\to 
	(\mathrm{LCH},\text{all}) $
	is a morphism of geometric setups, 
	we get a lax symmetric monoidal functor
	$$\Corr(\ProFin^\light)^\otimes
	\to  \Corr( \mathrm{LCH})^\otimes,$$ 
	which is indeed symmetric monoidal
	by \cite[2.3.6]{HM24}.
	\item [(2)]	Since $(\ProFin^\light,\text{all})\to 
	(\CondAn^\light,E) $
	is a morphism of geometric setups
	by \cite[3.4.11]{HM24}, 
	we get a lax symmetric monoidal functor
	$$\Corr(\ProFin^\light,\text{all})^\otimes
	\to  \Corr( \CondAn^\light,E )^\otimes,$$ 
	which is indeed symmetric monoidal
	by \cite[2.3.6]{HM24}.
\end{itemize}		
	
	\end{rem}

	Recall that a symmetric monoidal $\infty$-category
	$\C^\otimes$ is called \textit{presentably symmetric monoidal},
	if the underlying $\infty$-category $\C$
	is presentable, and the tensor functor
	$-\otimes-:\C\times \C\to \C $
	preserves colimits separately in each variable.
	
	\begin{prop}  [\text{\cite[3.8]{Nik16}}] \label{kan ext of lax sym mon functor}
		Let $q:\C^\otimes \to \D^\otimes $
		be a symmetric monoidal functor between symmetric monoidal
		$\infty$-categories $\C^\otimes$ and $\D^\otimes$.
		For any presentably symmetric monoidal $\infty$-category
		$\E^\otimes$, the functor
		$q^*:\Fun(\D,\E)^\otimes\to \Fun(\C,\E)^\otimes $
		induced by pre-composition with $q$
		admits a symmetric monoidal left adjoint
		$q_!:\Fun(\C,\E)^\otimes\to \Fun(\D,\E)^\otimes.$ 
	\end{prop}

\begin{prop} \label{profinite-lch 6ff}
	There is an adjunction
		\begin{center}
		\begin{tikzcd}
			6\mathrm{FF}( \ProFin^\light)
			\ar[r,shift left] &
			6\mathrm{FF}( \mathrm{LCH} ),
			\ar[l,shift left]
		\end{tikzcd}
	\end{center}
	where the right adjoint is given by restriction.
	And the left adjoint
	$ 6\mathrm{FF}(\ProFin^\light)\to 6\mathrm{FF}(\mathrm{LCH})$
	is fully faithful.
\end{prop}
\begin{proof}
	By Remark \ref{lax sym mon Corr},
we have the symmetric monoidal functor
$\Corr(\ProFin^\light)^\otimes
\to  \Corr( \mathrm{LCH} )^\otimes$.
Since $\Pr_\kappa^{L,\otimes}$ 
is a presentably symmetric monoidal $\infty$-category,
which is from \cite[2.3]{Aok25}, 
by Proposition \ref{kan ext of lax sym mon functor},
we know that the functor
$$
\Fun(\Corr( \mathrm{LCH} ),
{\Pr}^L_{\kappa})^\otimes
\to 
\Fun( \mathrm{Corr}(\ProFin^\light),
{\Pr}^L_{\kappa})^\otimes ,
$$
admits a symmetric monoidal left adjoint.
Therefore, we have the adjunction
\begin{center}
	\begin{tikzcd}
		\Fun(
		\mathrm{Corr}(\ProFin^\light),
		{\Pr}^L_{\kappa})^\otimes
		\ar[r,shift left] &
		\Fun(\Corr( \mathrm{LCH} ),
		{\Pr}^L_{\kappa})^\otimes,
		\ar[l,shift left]
	\end{tikzcd}
\end{center}
where the left adjoint is symmetric monoidal,
and the right adjoint is lax symmetric monoidal.
Passing to commutative algebras, we get the adjunction
\begin{center}
	\begin{tikzcd}
		\CAlg(\Fun(
		\mathrm{Corr}(\ProFin^\light),
		{\Pr}^L_{\kappa})^\otimes )
		\ar[r,shift left] &
		\CAlg(\Fun(\Corr( \mathrm{LCH} ),
		{\Pr}^L_{\kappa})^\otimes),
		\ar[l,shift left]
	\end{tikzcd}
\end{center}
which is exactly the adjunction
\begin{tikzcd}
	6\mathrm{FF}( \ProFin^\light)
	\ar[r,shift left] &
	6\mathrm{FF}(  \mathrm{LCH} ).
	\ar[l,shift left]
\end{tikzcd}
It's clear that the left adjoint
$ 6\mathrm{FF}(\ProFin^\light)\to 6\mathrm{FF}(\mathrm{LCH})$
is fully faithful.
\end{proof}

	The following result is given by \cite[4.18]{Zhu25}.
	But here we can relax some conditions.
	
	\begin{thm}\label{Zhu25 4.18}
		The six functor formalism
		$$\Shv(-;\Sp): \mathrm{Corr}(\ProFin^\mathrm{light})
		\to {\Pr}_\kappa^L$$
		is initial among all six functor formalisms
		$D:\mathrm{Corr}(\ProFin^\mathrm{light})
		\to  {\Pr}_\kappa^L$ satisfying:
		\begin{itemize}
			\item [(a)]  For any light profinite set $S$,
			the $\infty$-category $D(S)$ is dualizable,
			and the family $\{s^*:D(S)\to D(*)\}_{s\in S}$
			is jointly conservative.
			\item [(b)]
			$D:\ProFin^{\light,\op}\to  {\Pr}_\kappa^L$
			is a hypercomplete sheaf. 
		\end{itemize}
	In other words, if we let $6\mathrm{FF}(\ProFin^\mathrm{light})^\cont$
		be the full subcategory of 
		$6\mathrm{FF}(\ProFin^\mathrm{light})$
		spanned by those $D$ satisfying (a) and (b),
		then $\Shv(-;\Sp)\in  6\mathrm{FF}(\ProFin^\mathrm{light})^\cont $
		is initial.
		We denote the unique morphism
		by $\alpha:\Shv(-;\Sp)\to D(-)$.
	\end{thm}
\begin{proof}
	We will use the notation in \cite[Theorem 4.18]{Zhu25}. 
	Let $D\in 6\mathrm{FF}(\ProFin^\light)$
	satisfies the assumptions,
	whose image under the functor
$ 6\mathrm{FF}(\ProFin^\light)\to 6\mathrm{FF}(\mathrm{LCH})$
is still denoted by $D$.
Now since $\Shv(-;\Sp)\in \mathrm{CoSys}^c$
is initial,
there exists a unique morphism
$\alpha:\Shv(-;\Sp)\to D(-)$ in $\mathrm{CoSys}^c$.
In the proof of \cite[Theorem 4.18]{Zhu25},
Zhu showed that for any proper map $p:X\to Y$ of locally compact Hausdorff spaces, the commutative square
\begin{center}
	\begin{tikzcd}
		\Shv(Y;\Sp)\ar[r,"\alpha_Y"]\ar[d,"p^*"]
		& D(Y)\ar[d,"p^*"]\\
		\Shv(X;\Sp)\ar[r,"\alpha_X"]
		& D(X)
	\end{tikzcd}
\end{center}
is vertically right adjointable.
However, by Proposition \ref{profinite-lch 6ff},
it suffices to show that for any map $p:X\to Y$ of light profinite sets, the above commutative square is vertically right adjointable.
Now, if we drop the profinite descent condition in \cite[Definition 4.11]{Zhu25}, 
all things in the proof of \cite[Theorem 4.18]{Zhu25}
still work.
\end{proof}

\begin{rem}
	If we let $6\mathrm{FF}(\mathrm{LCH})^\cont$
	be the essential image of 
	$6\mathrm{FF}(\ProFin^\mathrm{light})^\cont$
	under the fully faithful functor
	$ 6\mathrm{FF}(\ProFin^\light)\hookrightarrow 6\mathrm{FF}(\mathrm{LCH})$,
	then there is an equivalence
	$$
	6\mathrm{FF}(\ProFin^\light)^\cont
	\stackrel{\sim}{\to} 6\mathrm{FF}(\mathrm{LCH})^\cont.
	$$
\end{rem}

\begin{rem}
	Let $f:S\to T$ be a map between light profinite sets.
There are two commutative squares
	\begin{center}
		\begin{tikzcd}
			\Shv(T;\Sp)\ar[r,"\alpha_T"]\ar[d,"f^*"]
			& D(T)\ar[d,"f^*"]\\
			\Shv(S;\Sp)\ar[r,"\alpha_S"]
			& D(S)
		\end{tikzcd}
		and
		\begin{tikzcd}
			\Shv(S;\Sp)\ar[r,"\alpha_S"]\ar[d,"f_!"]
			& D(S)\ar[d,"f_!"]\\
			\Shv(T;\Sp)\ar[r,"\alpha_T"]
			& D(T).
		\end{tikzcd}
	\end{center}
\end{rem}

	Next, we recall the concrete construction
	of the functor $\alpha_S:\Shv(S;\Sp)\to D(S)$,
	for any light profinite set $S$.
	\begin{cons} \label{cons of alpha}
		Let $S$ be a light profinite set.
		Note that the $\infty$-category $\Shv(S;\Sp)$
		is generated under colimits
		by objects of the form $i_{U!}1_{\Shv(U;\Sp)}$,
		where $i_U :U\hookrightarrow S$ is an open subset of $S$.
		Recall from \cite[3.10]{Zhu25} that the functor
		$$\alpha_S:\Shv(S;\Sp)\to D(S)$$
		is a colimit-preserving functor,
		which sends $ i_{U!}1_{\Shv(U;\Sp)} $
		to $i_{U!} 1_{D(U)}$,
		where $i_U: U\hookrightarrow S$
		is an open subset of $S$.
	\end{cons}

	With this construction, one has the following basic observation.
	\begin{rem} \label{sym mon for profinite set case}
		Let $S$ be a light profinite set.
	By the construction of the functor $\alpha_S$, we have:
		\begin{itemize}
			\item [(1)]
			$\alpha_S 1_{\Shv(S;\Sp) }\simeq 1_{D(S)}$. 
			\item [(2)] For any $A, B\in \Shv(S;\Sp)$,
			we have $\alpha_S(A)\otimes \alpha_S(B)\simeq 
			\alpha_S(A\otimes B)$.
			Indeed, it suffices to show this equivalence on generators.
			Note that for open subsets
			$i_{U}:U\hookrightarrow S$
			and  $i_{V}:V\hookrightarrow S$,
			we have
			$$
			i_{U!} 1_{ \Shv(U;\Sp) } \otimes i_{V! }1_{\Shv(V;\Sp)}
			\simeq i_{U!} i_U^* i_{V!}1_{\Shv(V;\Sp)}
			\simeq i_{U\cap V !} 1_{ \Shv(U\cap V;\Sp)},
			$$
			thus we are able to show the equivalence
			$$ \alpha_S(  i_{U!} 1_{ \Shv(U;\Sp) } \otimes i_{V! }1_{\Shv(V;\Sp)}  )
			\simeq  i_{U!} 1_{D(U)}\otimes i_{V!} 1_{D(V)}.
			$$
		\end{itemize}
	\end{rem}

	\section{Six functor formalisms on light condensed anima}
	
	In this section, we show that the six functor formalism
	$$\Shv(-;\Sp):\mathrm{Corr}(\CondAn^\mathrm{light},E)\to
	{\Pr}^L_{\kappa}$$
	is initial among all six functor formalisms satisfying some conditions, see Theorem \ref{initial}.

	We recall the extension result of six functor formalisms
	given by Heyer-Mann.
	\begin{thm} [\text{\cite[3.4.11]{HM24}}]
		There is a geometric class $E$ of maps of light condensed anima with the following properties:
		\begin{itemize}
			\item [(1)] There exists a unique six functor formalism
			$$\Shv(-;\Sp):\mathrm{Corr}(\CondAn^\light,E)\to
			{\Pr}^L$$
			extending the six functor formalism
			$\Shv(-;\Sp):\mathrm{Corr}(\ProFin^\light)\to \Pr^L$.
			\item [(2)] If $f:Y\to X$ is a map of light condensed anima such that for any map $S\to X$ from a light profinite set $S$, the pullback $Y\times_X S\to S$
			lies in $E$, then $f$ lies in $E$.
			\item [(3)] If $f:Y \to X$ is a map of light condensed anima which is $!$-locally on source or target in $E$,
			then $f$ lies in $E$.
			\item [(4)] If $f:X\to S$ lies in $E$ and $S$ is a light profinite set, then $X$ admits a 
			$\Shv(-;\Sp)^!$-cover by light profinite sets.
		\end{itemize}
		
		The maps in $E$ are called $!$-able maps.    
	\end{thm}

\begin{cor} \label{adjunction profinite and cond}
	There is an adjunction
	\begin{center}
			\begin{tikzcd}
			6\mathrm{FF}( \ProFin^\light)
			 \ar[r,shift left] &
			6\mathrm{FF}(  \CondAn^\light,E ),
			 \ar[l,shift left]
		\end{tikzcd}
	\end{center}
	where the right adjoint is given by restriction.
	And the left adjoint
	$ 6\mathrm{FF}( \ProFin^\light)
	\to 	6\mathrm{FF}(  \CondAn^\light,E )$
	is fully faithful.
\end{cor}

\begin{proof}
	By Remark \ref{lax sym mon Corr},
	we have the symmetric monoidal functor
	$\Corr(\ProFin^\light,\text{all})^\otimes
	\to  \Corr( \CondAn^\light,E )^\otimes$.
Since $\Pr_\kappa^{L,\otimes}$ 
is a presentably symmetric monoidal $\infty$-category,
which is from \cite[2.3]{Aok25}, 
 by Proposition \ref{kan ext of lax sym mon functor},
	we know that the functor
	 $$
	\Fun(\Corr(\CondAn^\light,E),
	{\Pr}^L_{\kappa})^\otimes
	\to 
	\Fun( \mathrm{Corr}(\ProFin^\light),
	{\Pr}^L_{\kappa})^\otimes ,
	$$
	admits a symmetric monoidal left adjoint.
	Therefore, we have the adjunction
		\begin{center}
		\begin{tikzcd}
		\Fun(
		\mathrm{Corr}(\ProFin^\light),
		{\Pr}^L_{\kappa})^\otimes
			\ar[r,shift left] &
			\Fun(\Corr(\CondAn^\light,E),
	{\Pr}^L_{\kappa})^\otimes,
			\ar[l,shift left]
		\end{tikzcd}
	\end{center}
where the left adjoint is symmetric monoidal,
and the right adjoint is lax symmetric monoidal.
	Passing to commutative algebras, we get the adjunction
		\begin{center}
		\begin{tikzcd}
		\CAlg(\Fun(
			\mathrm{Corr}(\ProFin^\light),
			{\Pr}^L_{\kappa})^\otimes )
			\ar[r,shift left] &
		\CAlg(\Fun(\Corr(\CondAn^\light,E),
				{\Pr}^L_{\kappa})^\otimes),
			\ar[l,shift left]
		\end{tikzcd}
	\end{center}
which is exactly the adjunction
	\begin{tikzcd}
		6\mathrm{FF}( \ProFin^\light)
		\ar[r,shift left] &
		6\mathrm{FF}(  \CondAn^\light,E ).
		\ar[l,shift left]
	\end{tikzcd}

\end{proof}

\begin{thm}\label{initial}
		The six functor formalism
	$$\Shv(-;\Sp):\mathrm{Corr}(\CondAn^\mathrm{light},E)\to
		{\Pr}^L_{\kappa} $$
	is initial among all six functor formalisms
	$D:\mathrm{Corr}(\CondAn^\mathrm{light},E)\to
		{\Pr}^L_{\kappa}$ satisfying:
	\begin{itemize}
		\item [(a)]
		For any light profinite set $S$,
		the $\infty$-category $D(S)$ is dualizable,
		and the family $\{s^*:D(S)\to D(*)\}_{s\in S}$
		is jointly conservative.
		\item [(b)]
		The functor $D:\CondAn^{\mathrm{light},\op}
		\to {\Pr}^L_{\kappa}$
		preserves limits.
	\end{itemize}
	In other words, if we let $6\mathrm{FF}(\CondAn^\mathrm{light},E)^\cont$
be the full subcategory of 
$6\mathrm{FF}(\CondAn^\mathrm{light},E)$
spanned by those $D$ satisfying (a) and (b),
then $\Shv(-;\Sp)\in  
6\mathrm{FF}(\CondAn^\mathrm{light},E)^\cont $
is initial.	
We denote the unique morphism by $\alpha:\Shv(-;\Sp)\to D(-).$
\end{thm}
\begin{proof}
By Corollary \ref{adjunction profinite and cond},
we have the fully faithful functor
$$	6\mathrm{FF}( \ProFin^\light)
\hookrightarrow 	6\mathrm{FF}(  \CondAn^\light,E ).
$$
Note that six functor formalisms
$D:\mathrm{Corr}(\CondAn^\mathrm{light},E)\to
{\Pr}^L_{\kappa} $
satisfying (a) and (b) exactly
lie in the essential image of
$	6\mathrm{FF}( \ProFin^\light)^\mathrm{cont}$
under this fully faithful functor.
	By Theorem \ref{Zhu25 4.18}, we know that 
$\Shv(-;\Sp)\in 	6\mathrm{FF}( \ProFin^\light)^\cont $
is initial.
Thus we can conclude that
$\Shv(-;\Sp)\in 6\mathrm{FF}(  \CondAn^\light,E ) $
is initial among all six functor formalisms
satisfying (a) and (b).
\end{proof}

\begin{rem}
From the proof of Theorem \ref{initial},
we know that there is an equivalence
	$$6\mathrm{FF}( \ProFin^\light)^\cont
\stackrel{\sim}{\to} 6\mathrm{FF}(  \CondAn^\light,E )^\cont.$$

\end{rem}

\begin{cor}
	Let $\mathscr{D}$ be a dualizable  $\infty$-category.
	The six functor formalism
	$$\Shv(-;\D):\mathrm{Corr}(\CondAn^\mathrm{light},E)\to
	{\Pr}^L_{\kappa} $$
	is initial among all six functor formalisms
	$D:\mathrm{Corr}(\CondAn^\mathrm{light},E)\to
	{\Pr}^L_{\kappa} $ satisfying:
	\begin{itemize}
		\item [(a)]
		For any light profinite set $S$,
		the $\infty$-category $D(S)$ is dualizable
		and the family $\{s^*:D(S)\to D(*)\}_{s\in S}$
		is jointly conservative.
		\item [(b)]
		The functor $D:\CondAn^{\mathrm{light},\op}\to 	{\Pr}^L_{\kappa}$
		preserves limits.
		\item [(c)] For each light condensed anima $X$,
		the category $D(X)$ is an algebra over $\D$.
	\end{itemize}
\end{cor}

\begin{proof}
Combine the proof in Theorem \ref{initial}
and \cite[4.21]{Zhu25}.
\end{proof}

Let $\Pr^\dual \subset \Pr^L$
be the full subcategory of $\Pr^L$
spanned by dualizable $\infty$-categories.	
	
\begin{prop}
	There is an adjunction
		\begin{center}
		\begin{tikzcd}
			\CAlg(\Pr^\dual)
			\ar[r,shift left] &
			6\mathrm{FF}(  \CondAn^\light,E )^\cont,
			\ar[l,shift left]
		\end{tikzcd}
	\end{center}
where the left adjoint is given by
$\C \mapsto \Shv(-;\C)$	
and the right adjoint is given by $D\mapsto D(*)$.
\end{prop}	
\begin{proof}
	It follows from Theorem \ref{initial}
	and \cite[4.19]{Zhu25}.
\end{proof}

\begin{rem}
From the adjunction
	\begin{tikzcd}
	\CAlg(\Pr^\dual)
	\ar[r,shift left] &
	6\mathrm{FF}(  \CondAn^\light,E )^\cont
	\ar[l,shift left]
\end{tikzcd},
we get the counit
$\epsilon: \Shv(-;D(*))\to D(-) $
in 	$	6\mathrm{FF}(  \CondAn^\light,E )^\cont$.
\end{rem}

	\begin{cor} \label{factorization}
		Let $D\in 6\mathrm{FF}(\CondAn^\light,E)^\cont$. 
		The map
		$\alpha:\Shv(-;\Sp)\to D(-)$ factors through
		\begin{center}
			\begin{tikzcd}
				\Shv(-;\Sp)\ar[rr,"\alpha"]\ar[rd]
				& & D(-)  \\
				& \Shv(-;D(*))\ar[ru,"\epsilon"swap] .&
			\end{tikzcd}
		\end{center}
	\end{cor}
\begin{proof}
	Applying counit to the map
	$\alpha:\Shv(-;\Sp)\to D(-)$, we get the commutative square
	\begin{center}
		\begin{tikzcd}
			\Shv(-;\Sp)\ar[r,"\sim"]\ar[d]
			& \Shv(-;\Sp)\ar[d,"\alpha"]\\
			\Shv(-;D(*))\ar[r,"\epsilon"]
			& D(-).
		\end{tikzcd}
	\end{center}
	Then it follows from the above square.
\end{proof}

\begin{cor} \label{factorization of epsilon}
	Let $D\in 6\mathrm{FF}(\CondAn^\light,E)^\cont$.
	The map $\epsilon: \Shv(-;D(*))\to D(-) $
	factors as
	$$
	\Shv(-;D(*)) 
	\simeq \Shv(-;\Sp)\otimes D(*)
	\stackrel{\alpha\otimes \mathrm{id}}{\longrightarrow}
	 D(-)\otimes D(*)\longrightarrow  D(-)
	$$
	in 	$6\mathrm{FF}(  \CondAn^\light,E )$.
\end{cor}	
	\begin{proof}
		Restricting $D$ to light profinite sets,
		by \cite[3.11]{Zhu25}
		and \cite[4.19]{Zhu25},
		we know that the counit
		$\epsilon: \Shv(-;D(*))\to D(-) $
		in $	6\mathrm{FF}( \ProFin^\light)^\mathrm{cont}$
		has the factorization
		$$
	\Shv(-;D(*)) 
	\simeq \Shv(-;\Sp)\otimes D(*)
	\stackrel{\alpha\otimes \mathrm{id}}{\longrightarrow}
	D(-)\otimes D(*)\longrightarrow  D(-).
	$$	
Thus what we want to show is followed by applying the functor
	$	6\mathrm{FF}( \ProFin^\light)
	\to	6\mathrm{FF}(  \CondAn^\light,E )$
to this factorization.		
	\end{proof}

	\begin{rem} \label{commutes with * & !}
		If $\alpha:D\to D^\prime $
		is a morphism in 
		$6\mathrm{FF}(\CondAn^\mathrm{light},E)$,
	then:
		\begin{itemize}
			\item [(1)] for any map $f:Y\to X$ of light condensed anima, there is a commutative square
			\begin{center}
				\begin{tikzcd}
					D(X)\ar[r,"\alpha_X"]\ar[d,"f^*"]
					& D^\prime(X)\ar[d,"f^*"]\\
					D(Y)\ar[r,"\alpha_Y"]
					& D^\prime (Y).
				\end{tikzcd}
			\end{center}
			\item [(2)] for any map $f:Y\to X$ in $E$, there is a commutative square
			\begin{center}
				\begin{tikzcd}
					D(Y)\ar[d,"f_!"]\ar[r,"\alpha_Y"]
					& D^\prime(Y) \ar[d,"f_!"]\\
					D(X)\ar[r,"\alpha_X"] & D^\prime (X).
				\end{tikzcd}
			\end{center}
		\end{itemize}
	\end{rem}

	\begin{rem} \label{commute with pullback}
		Let $X$ be a light condensed anima.
		We write $X\simeq \mathrm{colim}_i S_i$
		as the colimit of light profinite sets $S_i$.
		Suppose the maps $f_i:S_i \to X$
		and transition maps $f_{ij}:S_i \to S_j$.
		We have the commutative square
		\begin{center}
			\begin{tikzcd}
				\Shv(X;\Sp)\ar[d,"f_i^*"]\ar[r,"\alpha_X"]
				& D(X)\ar[d,"f_i^*"]\\
				\Shv(S_i;\Sp)\ar[r,"\alpha_{S_i}"]
				& D(S_i).
			\end{tikzcd}
		\end{center}
	\end{rem}

	\begin{cor} \label{preserves colimits and unit}
		Let $X$ be a light condensed anima
		and $D\in 6\mathrm{FF}(\CondAn^\light,E)^\cont$.
		\begin{itemize}
			\item [(1)]The functor $\alpha_X:\Shv(X;\Sp)\to D(X)$
			preserves all colimits.
			\item [(2)]We have 
			$ \alpha_X 1_{\Shv(X;\Sp) }\simeq 1_{D(X)}$.
			\item [(3)]
			For any $A,B\in \Shv(X;\Sp)$,
			we have 
			$\alpha_X(A)\otimes \alpha_X(B)\simeq \alpha_X (A\otimes B)$.
		\end{itemize}
	\end{cor}
	
	\begin{proof}
		We write $X\simeq \mathrm{colim}_i S_i $
		as the colimit of light profinite sets $S_i$,
		and denote maps $f_i:S_i\to X$.
		
		For (1), since the family $\{f_i^*:D(X)\to D(S_i) \}_i$
		is jointly conservative,
		it suffices to show that the functor
		$f_i^* \alpha_X $ preserves all colimits, for each $i$.
		Indeed, by Remark \ref{commute with pullback},
		we have $f_i^* \alpha_X \simeq \alpha_{S_i}f_i^*$,
		which preserves all colimits,
		since both $f_i^*$ and $\alpha_{S_i}$ preserve all colimits.
		
		For (2), note that we have
		\begin{align*}
			\alpha_X 1_{\Shv(X;\Sp) }
			&\simeq \lim_i f_{i*} f_i^* \alpha_X 1_{\Shv(X;\Sp) }
			\simeq \lim_i f_{i*} \alpha_{S_i} f_i^* 1_{\Shv(X;\Sp) }\\
			&\simeq \lim_i f_{i*} 1_{D(S_i)}
			\simeq  \lim_i f_{i*} f_i^* 1_{D(X)}
			\simeq 1_{D(X)}, 
		\end{align*}
			where the third equivalence is by Remark
		\ref{sym mon for profinite set case}.
		
		For (3), again, it suffices to show that
		$ f_i^* (\alpha_X(A)\otimes \alpha_X (B) )
		\simeq f_i^* \alpha_X(A\otimes B) $, for each $i$.
		Indeed, we have
		\begin{align*}
			f_i^* (\alpha_X(A)\otimes \alpha_X (B) )
			&\simeq 
			f_i^* \alpha_X(A)\otimes f_i^* \alpha_X (B) 
			\simeq \alpha_{S_i} f_i^* (A) \otimes \alpha_{S_i}f_i^*(B)\\
			&\simeq \alpha_{S_i} f_i^* (A\otimes B)
			\simeq  f_i^* \alpha_X (A\otimes B),
		\end{align*}
		where the third equivalence is by Remark
		\ref{sym mon for profinite set case}.

	\end{proof}

	\section{Applications}
	In this section, we give some applications. We will show that for any light condensed anima $X$ and 
	any presentable 6-functor formalism $D:  \mathrm{Corr}(\CondAn^\mathrm{light},E)\to
	{\Pr}^L_{\kappa}$ satisfying conditions (a) and (b) in Theorem \ref{initial}, we have:
	\begin{itemize}
		\item [(1)] The functor
		$f_X^*:D(X)\to D(X\times [0,1]) $
		is fully faithful, where $f_X:X\times [0,1]\to X$ is the projection,
		\item [(2)] The functor
		$f_X^*:D(X)\to D(X\times \mathbb{R}) $
		is fully faithful, where $f_X:X\times \mathbb{R}\to X$
		is the projection.
	\end{itemize}

	We first recall the category of kernels.
	\begin{rem}
		Let $D$ be a 3-functor formalism on a geometric setup
		$(\C,E)$.
		Recall from \cite[4.1.3]{HM24},
		there is a 2-\textit{category of kernels}
		$K_{D,S}$ associated with the 3-functor formalism $D$.
		For this category of kernels $K_{D,S} $,
		we have the following descriptions:
		\begin{itemize}
			\item  The objects in $K_{D,S}$
			are the morphisms $X\to S$ in $E$.
			\item Given two objects $X,Y\in K_{D,S}$,
			the category
			$\Fun_{K_{D,S} }(Y,X)$
			of homomorphisms from $Y$ to $X$
			is given by
			$$\Fun_{K_{D,S} }(Y,X)\simeq D(X\times_S Y) .$$
			\item Given three objects $X,Y ,Z$ in $K_{D,S}$,
			and morphisms
			$M: Y\to X $, $N:Z\to Y$ in $K_{D,S}$,
			the composition
			$M\circ N:Z\to X $
			is given by
			$$
			M\circ N\simeq \pi_{XZ!}( \pi_{XY}^* M \otimes
			\pi_{YZ}^* N) \in D(X \times_S Z).
			$$
		\end{itemize}
	\end{rem}

	We then recall the concept of suave and prim objects from 
	\cite[Section 4.4]{HM24}.
	\begin{dfn}
		Let $D$ be a 3-functor formalism on a geometric setup
		$(\C,E)$ and $f:X\to S$ a morphism in $E$.
		We fix an object $P\in D(X)$.
		\begin{itemize}
			\item [(1)] We say $P$ is $f$-\textit{suave},
			if $P\in D(X)$, viewed as a morphism $X\to S$ in 
			$K_{D,S}$, is a left adjoint.
			We denote its right adjoint
			$S\to X$ by $\mathrm{SD}_f(P)\in D(X)$,
			which is called the $f$-\textit{suave dual} of $P$.
			\item [(2)] We say $P$ is $f$-\textit{prim},
			if $P\in D(X)$, viewed as a morphism $X\to S$ in 
			$K_{D,S}$, is a right adjoint.
			We denote its left adjoint
			$S\to X$ by $\mathrm{PD}_f(P)\in D(X)$,
			which is called the $f$-\textit{prim dual} of $P$.
		\end{itemize}
	\end{dfn}

	\begin{rem} \label{phi_alpha}
		Let $\alpha:D\to D^\prime$ be a morphism of six functor formalisms on a geometric setup $(\C,E)$.
		Recall from \cite[4.2.1]{HM24} that
		for any $S\in \C$, there is an induced 2-functor
		$\varphi_\alpha: K_{D,S}\to K_{D^\prime,S}$,
		which acts on objects as the identity, and on morphisms categories as
		$$ \alpha_{X\times_S Y}:
		D(X\times_S Y)\to  D^\prime(X\times_S Y).
		$$
		
		Recall from \cite[4.1.5]{HM24} that there is a 2-functor
		$ \Psi_{D,S}: \K_{D,S}\to \Cat_2$,
		which sends $X\in K_{D,S}$
		to $D(X)\in \Cat_2$,
		and a morphism $M\in \Fun_S(Y,X)\simeq D(X\times_S Y)$
		to the functor
		$$
		\Psi_{D,S}(M)= \pi_{X!}
		(M\otimes \pi_Y^*(-) ):D(Y)\to (X),
		$$
		where $\pi_X: X\times_S Y \to X$
		and $\pi_Y:X\times_S Y\to Y$.
		
		Now, for any morphism $\alpha:D\to D^\prime$
		between six functor formalisms on a geometric setup $(\C,E)$,
		there is an induced morphism
		$$ \Phi_{D,S}\to \Psi_{D^\prime,S}\circ \varphi_\alpha$$
		between the 2-functors
		from $K_{D,S}$ to $\Cat_2$.
	\end{rem}

	\begin{lem} \label{preserves SD and PD}
		Let $f:X\to S$ be a map in $E$ and
		$\alpha:D\to D^\prime$ a morphism between six functor formalisms.
		Fix an object $P\in D(X)$. We have:
		\begin{itemize}
			\item [(1)]
			If $P$ is $f$-suave, then
			$\alpha_X(P)\in D^\prime(X)$
			is $f$-suave, and
			$\alpha_X( \mathrm{SD}_f(P) )\simeq
			\mathrm{SD}_f( \alpha_X(P)). $
			\item [(2)]
			If $P$ is $f$-prim, then
			$\alpha_X(P)\in D^\prime(X)$
			is $f$-prim, and
			$\alpha_X( \mathrm{PD}_f(P) )\simeq
			\mathrm{PD}_f( \alpha_X(P)). $
		\end{itemize}
	\end{lem}
	\begin{proof}
		The morphism $\alpha:D\to D^\prime$
		induces a 2-functor
		$\varphi_\alpha: K_{D,S}\to K_{D^\prime,S}$,
		which preserves adjunctions.
		If $P$ is $f$-suave,
		then $P\in D(X)$, viewed as a morphism
		$X\to S$ in $K_{D,S}$, is left adjoint.
		Thus by Remark \ref{phi_alpha},
		 the 2-functor $\varphi_\alpha$ sends
		the adjunction
		\begin{tikzcd}
			X \ar[r,shift left,"P"] &
			S \ar[l,shift left,"\mathrm{SD}_f(P)"]
		\end{tikzcd}
		to the adjunction
		\begin{tikzcd}
			X \ar[r,shift left,"\alpha_X(P)"] &
			S \ar[l,shift left,"\alpha_X(\mathrm{SD}_f(P))"]
		\end{tikzcd}.
		Hence $\mathrm{SD}_f( \alpha_X(P) ) 
		\simeq \alpha_X(\mathrm{SD}_f(P)) $.
		The case $P$ being $f$-prim is similar. 
	\end{proof}

	Next, we recall the definition of suave maps and prim maps
	from \cite[Section 4.5]{HM24}.
	\begin{dfn}
		Let $D$ be a three functor formalism on a geometric setup
		$(\C,E)$, and $f:X\to S$ a morphism in $E$.   
		\begin{itemize}
			\item [(1)] We say $f$ is $D$-\textit{suave},
			if $1_{D(X)}\in D(X)$ is $f$-suave.
			We denote $\omega_f:=\mathrm{SD}_f(1_{D(X)} )
			\simeq f^! 1_{D(S)}\in D(X)$,
			which is called the \textit{dualizing complex} of $f$.
			\item [(2)] We say $f$ is $D$-\textit{prim},
			if $1_{D(X)}\in D(X)$ is $f$-prim.
			We denote $D_f:=\mathrm{PD}_f(1_{D(X)} )
			\simeq p_{1*}\Delta_! 1_{D(X)} \in D(X)$,
			which is called the \textit{codualizing complex} of $f$.
		\end{itemize}
	\end{dfn}

	In Lemma \ref{preserves SD and PD},
	if we take $P\simeq 1_{D(X)}\in D(X)$,
	and take $\alpha:D\simeq \Shv(-;\Sp)\to D^\prime$
	to be the unique morphism from Theorem \ref{initial}, 
	then we get:
	\begin{cor} \label{preserves delta_f and D_f}
		Let $f:X\to S$ be a map in $E$
		and $\alpha:D\simeq \Shv(-;\Sp) \to D^\prime$ the unique morphism from Theorem \ref{initial}. We have:
		\begin{itemize}
			\item [(1)]
			If $f$ is $D$-suave, then $f$ is $D^\prime$-suave,
			and $\alpha_X \omega_f \simeq \omega_f^\prime $,
			where $\omega_f\simeq f^! 1_{D(S)}$
			and $\omega_f^\prime\simeq f^! 1_{D^\prime(S)}$.
			\item [(2)]
			If $f$ is $D$-prim, then $f$ is $D^\prime$-prim,
			and $\alpha_X D_f \simeq D_f^\prime$,
			where $D_f\simeq p_{1*}\Delta_!1_{D(X)}$
			and $D_f^\prime  \simeq p_{1*}\Delta_!1_{D^\prime(X)}.$
		\end{itemize}
	\end{cor}
	\begin{proof}
		By Corollary \ref{preserves colimits and unit},
		we have $\alpha_X( 1_{D(X)} )\simeq 1_{D^\prime(X)}$.
		Now, it follows from Lemma \ref{preserves SD and PD}.
	\end{proof}

	In the proof of Theorem \ref{suave prim case commutes},
	we will use the concept of 2-natural transformation.
	So we recall this concept.
	\begin{dfn}[2-natural transformation]
		Let $D,E:\A\to \B$ be 2-functors between 2-categories.
		A \textit{2-natural transformation}
		$\alpha:D\implies E:\A\to \B$
		between $D$ and $E$ assigns to
		 each object $A\in \A$,
		an arrow $\alpha_A: DA \to EA $ in $\B$,
		such that:
		\begin{itemize}
			\item [(1)] for each map $f:A_1\to A_2$ in $\A$,
			we have
			$ \alpha_{A_2}\circ Df
			\simeq    Ef\circ \alpha_{A_1} $ in $\B$. 
			\item [(2)] for each 2-cell $\mu:f\implies g$
			from $A_1$ to $A_2$ in $\A$,
			we have
		$ \alpha_{A_2}\circ D\mu
		\simeq    E\mu \circ \alpha_{A_1} $
		in $\B$.
		\end{itemize}
	\end{dfn}

	\begin{thm} 	\label{suave prim case commutes}
	Let $f:Y\to X$ be a map in $E$
	and $\alpha:D\simeq \Shv(-;\Sp) \to D^\prime$ the unique morphism from Theorem \ref{initial}. We have:
	\begin{itemize}
		\item [(1)]If $f:Y\to X$ is $D$-suave,
		then the commutative square
		\begin{center}
			\begin{tikzcd}
				D(Y)\ar[r,"\alpha_Y"]\ar[d,"f_!"]
				& D^\prime(Y)\ar[d,"f_!"]\\
				D(X)\ar[r,"\alpha_X"] & D^\prime(X)
			\end{tikzcd}
		\end{center}
		is vertically right adjointable. That is, the canonical map
		$ \alpha_Y f^! \to f^! \alpha_X$
		is an equivalence.
		\item [(2)]If $f:Y\to X$ is $D$-prim,
		then the commutative square
		\begin{center}
			\begin{tikzcd}
				D(X)\ar[r,"\alpha_X"]\ar[d,"f^*"]
				& D^\prime(X)\ar[d,"f^*"]\\
				D(Y)\ar[r,"\alpha_Y"] & D^\prime(Y)
			\end{tikzcd}
		\end{center}
		is vertically right adjointable. That is, the canonical map
		$ \alpha_X f_* \to  f_* \alpha_Y$
		is an equivalence.
	\end{itemize}
\end{thm}
	\begin{proof}
		In each of the cases, we can construct 
		some isomorphism between the two functors.
		For (1), we have
		$$
		\alpha_Y f^!
		\simeq \alpha_Y( f^! 1_{D(X)}  \otimes f^*(-) )
		\simeq \alpha_Y f^! 1_{D(X)} \otimes \alpha_Y f^*(-)
		\simeq f^! 1_{D^\prime(X)} \otimes f^* \alpha_X(-)
		\simeq f^!\alpha_X,
		$$  
		where the first equivalence is because $f$ is $D$-suave,
		the second equivalence is by
		Corollary \ref{preserves colimits and unit},
		the third equivalence is by
		Corollary \ref{preserves delta_f and D_f}, 
		and the last equivalence is because $f$ is $D^\prime$-suave.
		
		And for (2), we have
		\begin{align*}
			\alpha_X f_*
			&\simeq \alpha_X f_!(-\otimes D_f)
			\simeq f_! \alpha_Y(-\otimes D_f)
			\simeq f_!( \alpha_Y(-)\otimes \alpha_Y D_f )
			\simeq
			f_!( \alpha_Y(-)\otimes D_f^\prime)
			\simeq f_* \alpha_Y,
		\end{align*}
		where the first equivalence is because $f$ is $D$-prim,
		the second equivalence is by 
		Remark \ref{commutes with * & !}, 
		the third equivalence is by
		Corollary \ref{preserves colimits and unit}, 
		the fourth equivalence is by
		Corollary \ref{preserves delta_f and D_f},
		and the last equivalence is because $f$ is $D^\prime$-prim.
		
		However, with this strategy it is not clear why the constructed isomorphism is the natural one.
		
		Now, we mimic a conceptual proof
		from \cite[4.5.13]{HM24}
		that the isomorphism is indeed the correct one, using the 2-category of kernels.
		
		We first prove (1).
		We claim that
		$ \Psi_{D,X}\to \Psi_{D^\prime,X}\circ \varphi_\alpha$
		is a 2-natural transformation
		between 2-functors from
		$K_{D,X}$ to $\Cat_2$.
		This is ensured by (a) and (b) in the following:
		
		(a). For any map $M:Y_1\to Y_2$ in $K_{D,X}$,
		i.e. any object $ M\in D(Y_2 \times_X Y_1)$,
		there is a commutative square
	\begin{center}
		\begin{tikzcd}
			D(Y_1)\ar[r,"\alpha_{Y_1} "]
			\ar[d,"\pi_{2!}(\pi_1^*(-)\otimes M)"swap]
			& D^\prime(Y_1)
			\ar[d,"\pi_{2!}(\pi_1^*(-)
			\otimes \alpha_{Y_2\times_X Y_1}( M) )"]\\
			D(Y_2)\ar[r,"\alpha_{Y_2}"]
			& D^\prime(Y_2).
		\end{tikzcd}
	\end{center}
	Indeed, this square can be written as
\begin{center}
	\begin{tikzcd}
		D(Y_1)\ar[r,"\pi_1^*"]\ar[d,"\alpha_{Y_1}"]
		& D(Y_2\times_X Y_1)\ar[rr,"-\otimes M"]
		\ar[d,"\alpha_{Y_2\times_X Y_1}"]
		& & D(Y_2\times_X Y_1)\ar[r,"\pi_{2!}"]
		\ar[d,"\alpha_{Y_2\times_X Y_1}"]
		& D(Y_2) \ar[d,"\alpha_{Y_2}"]\\
		D^\prime(Y_1)\ar[r,"\pi_1^*"]
		& D^\prime(Y_2\times_X Y_1)
		\ar[rr,"-\otimes \alpha_{Y_2\times_X Y_1}(M) "]
		&& D^\prime(Y_2\times_X Y_1)\ar[r,"\pi_{2!}"]
		& D^\prime(Y_2),
	\end{tikzcd}
\end{center}
where the above squares are commutative
by Remark \ref{commutes with * & !} and
Corollary \ref{preserves colimits and unit}.

(b). Let $M$ and $N$ be any two morphisms from $Y_1$ to $Y_2$
in $K_{D,X}$. For any 2-morphism $\mu:M\to N$ in $K_{D,X}$,
i.e. a morphism $\mu:M\to N$ in $D(Y_2\times_X Y_1)$,
it is easy to show that
$$
\alpha_{Y_2}
\pi_{2!}( \pi_1^*(-)\otimes \mu )
\simeq 
\pi_{2!}(\pi_1^*\alpha_{Y_1}(-)
\otimes \alpha_{Y_2\times_X Y_1}( \mu) ).
$$

Now if $f:Y\to X$ is $D$-suave, then
after applying the natural transformation
$ \Psi_{D,X}\to \Psi_{D^\prime,X}\circ \varphi_\alpha:
K_{D,X}\to \Cat_2$ to the adjunction
		\begin{tikzcd}
			Y \ar[r,shift left,"1_{D(Y)} "] &
			X \ar[l,shift left,"\omega_f"]
		\end{tikzcd}
		in $K_{D,X}$,
		we get the commutative squares
		\begin{center}
			\begin{tikzcd}
				D(Y)\ar[d,"f_!"]\ar[r,"\alpha_Y"]
				& D^\prime(Y)\ar[d,"f_!"]\\
				D(X)\ar[r,"\alpha_X"]
				& D^\prime(X)
			\end{tikzcd}
	and
			\begin{tikzcd}
				D(X)\ar[r,"\alpha_X"]\ar[d,"\omega_f \otimes f^*"]
				& D^\prime(X)\ar[d,"\omega_f^\prime \otimes f^*"]\\
				D(Y)\ar[r,"\alpha_Y"] & D^\prime(Y)
			\end{tikzcd}
		\end{center}
		in $\Cat_2$.
		By \cite[2.5]{KS74},
		the above two squares are mates to each other.
		Thus
	 the commutative square
	\begin{center}
		\begin{tikzcd}
			D(Y)\ar[r,"\alpha_Y"]\ar[d,"f_!"]
			& D^\prime(Y)\ar[d,"f_!"]\\
			D(X)\ar[r,"\alpha_X"] & D^\prime(X)
		\end{tikzcd}
	\end{center}
	is vertically right adjointable.
		
		The proof of (2) is similar
		if we run the argument in $K_{D,X}^{\mathrm{co},\op}$.
	\end{proof}

	\begin{lem} \label{after tensor ver right adj}
		Suppose a commutative square
		\begin{center}
			\begin{tikzcd}
				\C\ar[r,"h"]\ar[d,"f"]
				& \C^\prime \ar[d,"f^\prime"]\\
				\D \ar[r,"g"] & \D^\prime
			\end{tikzcd}
		\end{center}
in $\Pr^L$,	which is vertically	right adjointable.
If both $f^R$ and $f^{\prime R}$
preserve colimits,
then for any $\E \in \Pr^L$,
the square
	\begin{center}
	\begin{tikzcd}
		\C\otimes \E \ar[r,"h\otimes \mathrm{id}_\E"]
		\ar[d,"f\otimes \mathrm{id}_\E "]
		& \C^\prime \otimes \E
		 \ar[d,"f^\prime \otimes \mathrm{id}_\E"]\\
		\D \otimes \E \ar[r,"g\otimes \mathrm{id}_\E"]
		 & \D^\prime \otimes\E
	\end{tikzcd}
\end{center}
is also vertically right adjointable.
	\end{lem}
\begin{proof}
	Since $f^R$ and $f^{\prime R}$ preserve colimits,
	we have $(f\otimes\mathrm{id}_\E)^R\simeq f^R \otimes \mathrm{id}_\E $ and
	$(f^\prime\otimes\mathrm{id}_\E)^R\simeq f^{\prime R}  \otimes \mathrm{id}_\E .$
Thus, applying $-\otimes \mathrm{id}_\E$
to the canonical equivalence
$h f^R \stackrel{\sim}{\to}  f^{\prime R} g$,
we get the canonical equivalence
$$
(h\otimes \mathrm{id}_\E)\circ 
 (f\otimes\mathrm{id}_\E)^R
 \stackrel{\sim}{\to}
 (f^\prime\otimes\mathrm{id}_\E)^R \circ 
 (g\otimes \mathrm{id}_\E),
$$
which finishes the proof.
\end{proof}

\begin{cor} \label{Shv(-;D(*))to D(-) case}
	Let $D\in 6\mathrm{FF}(\CondAn^\light,E)^\cont$.
	Let $f:Y\to X$ be a map in $E$
	and consider the counit $\epsilon: \Shv(-;D(*))\to D(-) $.
	We have:
	\begin{itemize}
		\item [(1)] If $f:Y\to X$ is $\Shv(-;\Sp)$-suave,
		then the commutative square
		\begin{center}
			\begin{tikzcd}
				\Shv(Y;D(*))\ar[r,"\epsilon_Y"]\ar[d,"f_!"]
				& D(Y)\ar[d,"f_!"]\\
				\Shv(X;D(*))\ar[r,"\epsilon_X"]
				& D(X)
			\end{tikzcd}
		\end{center}
		is vertically right adjointable.
		\item [(2)] If $f:Y\to X$ is $\Shv(-;\Sp)$-prim,
		then the commutative square
		\begin{center}
			\begin{tikzcd}
				\Shv(X;D(*))\ar[r,"\epsilon_X"]\ar[d,"f^*"]
				& D(X)\ar[d,"f^*"]\\
				\Shv(Y;D(*))\ar[r,"\epsilon_Y"]
				& D(Y)
			\end{tikzcd}
		\end{center}
		is vertically right adjointable.
	\end{itemize}
\end{cor}
\begin{proof}
	Note that the map
	$\sigma: D^\prime(-) \simeq D(-)\otimes D(*)\to D(-) $
	is a morphism of six functor formalisms,
	which induces a 2-functor
	$\varphi_\sigma:K_{D^\prime,X}\to K_{D,X}$
	by Remark \ref{phi_alpha}.
	Thus we get a map
	$ \Psi_{D^\prime,X}\to \Psi_{D,X}\circ \varphi_\sigma $
	between 2-functors from
	$K_{D^\prime,X}$ to $\Cat_2$.
	We claim that
	 $ \Psi_{D^\prime,X}\to \Psi_{D,X}\circ \varphi_\sigma $
	is a 2-natural transformation,
	which is ensured by (a) and (b) in the following:
	
	(a). For any map $M:Y_1 \to Y_2$ in $K_{D^\prime,X}$, 
 i.e. an object $M\in \Fun_{K_{D^\prime,X}}(Y_1,Y_2  ) 
  \simeq D^\prime(Y_2 \times_X Y_1)$, there is a commutative square
  \begin{center}
  	\begin{tikzcd}
  		D^\prime(Y_1)\ar[r,"\sigma_{Y_1}"]
  		\ar[d," \pi_{2!}(\pi_1^*(-)\otimes M ) "swap]
  		& D(Y_1)
  		\ar[d," \pi_{2!}(\pi_1^*(-)\otimes
  		 \sigma_{Y_2\times_X Y_1} (M) ) "]\\
  		D^\prime(Y_2)\ar[r,"\sigma_{Y_2}"]
  		& D(Y_2).
  	\end{tikzcd}
  \end{center}
 Note that this square can be written as
  \begin{center}
  	\begin{tikzcd}
  			D^\prime(Y_1)\ar[r,"\pi_1^*"]\ar[d,"\sigma_{Y_1}"]
  			& D^\prime(Y_2 \times_X Y_1 )
  			\ar[rr,"-\otimes M"]
  			\ar[d,"\sigma_{Y_2 \times_X Y_1 }"]
  			&& D^\prime(Y_2 \times_X Y_1 )\ar[r,"\pi_{2!}"]
  			\ar[d,"\sigma_{Y_2 \times_X Y_1 }"]
  			& D^\prime(Y_2)\ar[d,"\sigma_{Y_2}"]\\
  			D(Y_1) \ar[r,"\pi_1^*"]
  			&  D(Y_2 \times_X Y_1 )
  			\ar[rr,"-\otimes\sigma_{Y_2 \times_X Y_1} (M) "]
  			&& D(Y_2 \times_X Y_1 )\ar[r,"\pi_{2!}"]
  			& D(Y_2)
  	\end{tikzcd}
  \end{center}
  where each square is commutative.
  
 (b). Let $M$ and $N$ be any two morphisms from $Y_1$ to $Y_2$
 in $K_{D^\prime,X}$. For any 2-morphism 
 $\mu:M\to N$ in $K_{D^\prime,X}$,
 i.e. a morphism $\mu:M\to N$ in $D^\prime(Y_2\times_X Y_1)$,
 it is easy to show that
 $$
 \sigma_{Y_2} \pi_{2!}(\pi_1^*(-)\otimes \mu )
 \simeq
 \pi_{2!}(\pi_1^* \sigma_{Y_1}(-)\otimes
  \sigma_{Y_2\times_X Y_1}(\mu)   ).
 $$

Now, by Corollary \ref{factorization of epsilon}, 
the square in (1) can be written as
  \begin{center}
  	\begin{tikzcd}
  		\Shv(Y;\Sp)\otimes D(*) \ar[r,"\alpha_Y \otimes \mathrm{id}"]\ar[d]
  		& D(Y)\otimes D(*)\ar[r]\ar[d]
  		& D(Y)\ar[d]\\
  		\Shv(X;\Sp)\otimes D(*)
  		\ar[r,"\alpha_X \otimes \mathrm{id} "]
  		& D(X)\otimes D(*)\ar[r]
  		& D(X).
  	\end{tikzcd}
  \end{center}
 In order to show the square in (1) is vertically right adjointable,
 it suffices to show the left square and right square in the above are both vertically right adjointable. 
  
 If $f:Y\to X$ is $\Shv(-;\Sp)$-suave,
 then $f^! : \Shv(X;\Sp)\to \Shv(Y;\Sp)$
 and $f^!:D(X)\to D(Y)$ preserve all colimits. 
 By Theorem \ref{suave prim case commutes}
 and Lemma \ref{after tensor ver right adj},
 we know that the left square
 	\begin{center}
 	\begin{tikzcd}
 		\Shv(Y;\Sp)\otimes D(*)
 		 \ar[r,"\alpha_Y \otimes \mathrm{id} "]
 		 \ar[d,"f_!\otimes \mathrm{id}"]
 		& D(Y)\otimes D(*) \ar[d,"f_!\otimes \mathrm{id}"]\\
 		\Shv(X;\Sp)\otimes D(*) 
 		\ar[r,"\alpha_X \otimes \mathrm{id}"] & D(X)\otimes D(*)
 	\end{tikzcd}
 \end{center}
 is vertically right adjointable. 
  For the right square to be vertically right adjointable,
 after applying the natural transformation
  $ \Psi_{D^\prime,X}\to \Psi_{D,X}\circ \varphi_\sigma :
  K_{D^\prime,X}\to \Cat_2  $
  to the adjunction
  	\begin{tikzcd}
  	Y \ar[r,shift left,"1_{D^\prime(Y)} "] &
  	X \ar[l,shift left,"\omega_f"]
  \end{tikzcd}
  in $K_{D^\prime,X}$, we get two commutative squares
  \begin{center}
  	\begin{tikzcd}
  		D(Y)\otimes D(*)\ar[r]\ar[d,"f_!\otimes \mathrm{id}"]
  		& D(Y)\ar[d,"f_!"]\\
  		D(X)\otimes D(*)\ar[r]
  		& D(X)
  	\end{tikzcd}
  	and
  		\begin{tikzcd}
  		D(X)\otimes D(*)\ar[r]
  		\ar[d," \omega_f \otimes f^*\otimes \mathrm{id} "]
  		& D(X)\ar[d," \omega_f \otimes f^* "]\\
  		D(Y)\otimes D(*)\ar[r]
  		& D(Y),
  	\end{tikzcd}
  \end{center}
  which are mates to each other
  by \cite[2.5]{KS74}.
  Thus we can conclude that the square
  \begin{center}
  		\begin{tikzcd}
  		D(Y)\otimes D(*)\ar[r]\ar[d,"f_!\otimes \mathrm{id}"]
  		& D(Y)\ar[d,"f_!"]\\
  		D(X)\otimes D(*)\ar[r]
  		& D(X)
  	\end{tikzcd}
  \end{center}
  is vertically right adjointable,
  which finishes the proof of (1).
  
  The proof of (2) is similar.
  
\end{proof}

	Recall that given a six functor formalism
	$D:\mathrm{Corr}(\CondAn^\mathrm{light},E)\to {\Pr}^L$,
	for any $!$-able map $f:X\to *$ and any object
	 $E\in D(*)$, one has:
	\begin{itemize}
		\item [(1)](sheaf cohomology)
		$H^*(X,E)_D:= f_* f^* E$.
		\item [(2)](compactly supported cohomology)
		$H^*_c(X,E)_D:= f_! f^* E$.
		\item [(3)](sheaf homology)
		$ H_*(X,E)_D:= f_! f^!E$.
		\item [(4)](Borel-Moore homology)
		$H_*^\mathrm{BM}(X,E)_D:=f_* f^! E$.
	\end{itemize}
	Here, we write down the six functor formalism $D$
	to dedicate the dependence on it.

	\begin{prop}
	Let $D\in 6\mathrm{FF}(\CondAn^\light,E)^\cont$.
		For the map $ \epsilon:\Shv(-;D(*))\to D(-)$,
		we have:
		\begin{itemize}
			\item [(1)] If $f:X\to *$ is $!$-able, then
			$ H^*_c(X,E)_{\Shv(-;D(*))}\simeq H^*_c(X,E)_D$.
			\item [(2)]
			If $f:X\to *$ is $!$-able, then
			$H_*^\mathrm{BM}(X,E)_{\Shv(-;D(*))}
			\simeq H_*^\mathrm{BM}(X,E)_D $.
			\item [(3)]
			If $f:X\to *$ is $\Shv(-;D(*))$-prim,
			then $ H^*(X,E)_{\Shv(-;D(*))}\simeq H^*(X,E)_D.$
			\item [(4)]
			If $f:X\to *$ is $\Shv(-;D(*))$-suave,
			then $ H_*(X,E)_{\Shv(-;D(*))}\simeq H_*(X,E)_D.$
		\end{itemize}  
	\end{prop}
	\begin{proof}
		Note that by Remark \ref{commutes with * & !},
		the map $ \epsilon:\Shv(-;D(*))\to D(-)$
		gives the commutative diagram
		\begin{center}
			\begin{tikzcd}
				\Shv(*;D(*))\ar[r,"\sim"]\ar[d,"f^*"]
				&  D(*)\ar[d,"f^*"]\\
				\Shv(X;D(*))\ar[r]\ar[d,"f_!"]
				& D(X)\ar[d,"f_!"]\\
				\Shv(*;D(*))\ar[r,"\sim"]
				&  D(*).
			\end{tikzcd}
		\end{center}
		It follows that (1) holds.
		Passing to right adjoints, we know that (2) holds.
		For (3),
		if $f:X\to *$ is $\Shv(-;D(*))$-prim,
		by Corollary \ref{Shv(-;D(*))to D(-) case},
		we get the commutative diagram
		\begin{center}
			\begin{tikzcd}
				\Shv(*;D(*))\ar[r,"\sim"]\ar[d,"f^*"]
				&  D(*)\ar[d,"f^*"]\\
				\Shv(X;D(*))\ar[r]\ar[d,"f_*"]
				& D(X)\ar[d,"f_*"]\\
				\Shv(*;D(*))\ar[r,"\sim"]
				&  D(*).
			\end{tikzcd}
		\end{center}
		Thus we have $ H^*(X,E)_{\Shv(-;D(*))}\simeq H^*(X,E)_D.$
		Using  Corollary \ref{Shv(-;D(*))to D(-) case},
		we can also prove (4) similarly.
	\end{proof}

	\begin{prop} \label{f_X^* fully faithful}
		Let $I=[0,1]\subset \mathbb{R}$
		be the unit interval and $f:I \to *$ the projection.
	Let $D\in 6\mathrm{FF}(\CondAn^\light,E)^\cont$,   
		and $X$ a light condensed anima with the projection $f_X:X\times I \to X$. 
		Then the pullback functor
		$ f_X^*:D(X)\to D(X\times I) $
		is fully faithful.
		In particular,
		$f_{X*} 1_{D(X\times I)}\simeq 1_{D(X)}$.
	\end{prop}
	\begin{proof}
		By \cite[4.8.7]{HM24},
		the functor $f_X^*:\Shv(X;\Sp)\to \Shv(X\times I;\Sp)$
		is fully faithful.
		Since $D(*)$ is a dualizable category,
		the functor $ -\otimes D(*)$ preserves fully faithful functors.
		Thus the functor
		$f_X^*:\Shv(X;D(*)  )\to \Shv(X\times I;D(*) )$
		is also fully faithful.
		Since $f: I \to *$
		is $\Shv(-;\Sp)$-prim, 
		by Corollary \ref{Shv(-;D(*))to D(-) case},  
		we have the commutative diagram
		\begin{center}
			\begin{tikzcd}
				\Shv(*;D(*))\ar[r,"\sim"]\ar[d,"f^*"]
				&  D(*)\ar[d,"f^*"]\\
				\Shv( I;D(*))\ar[r]\ar[d,"f_{*}"]
				& D( I)\ar[d,"f_{*}"]\\
				\Shv(*;D(*))\ar[r,"\sim"]
				&  D(*).
			\end{tikzcd}
		\end{center}
		Thus we have
		$f_* 1_{D(I)}\simeq f_* f^* 1_{D(*)}
		\simeq  1_{D(*)}. $
		Since $f:I\to *$ is $\Shv(-;\Sp)$-prim,
		by Corollary \ref{preserves delta_f and D_f},
		we know that $f:I\to *$ is $D$-prim.
		Applying base change to the cartesian square
		\begin{center}
			\begin{tikzcd}
				X\times I \ar[r,"p_I"]\ar[d,"f_X"]
				& I \ar[d,"f"]\\
				X \ar[r,"p"] & *,
			\end{tikzcd}
		\end{center}
	 we get
		$$ f_{X*}1_{D(X\times I)}
		\simeq  f_{X*}p_{I}^* 1_{D(I)}
		\simeq p^* f_* 1_{D(I)} 
		\simeq  p^* 1_{D(*)}
		\simeq 1_{D(X)}. $$
		Note that as the pullback of $f:I\to *$,
		the map $f_X:X\times I \to X$ is also $D$-prim.
		Now, for any $E\in D(X)$,
		since $f_X:X\times I \to X$ is $D$-prim, 
		by the projection formula, we get
		$$
		f_{X*} f_X^* E
		\simeq f_{X*}( f_X^* E\otimes  1_{D(X\times I)} )
		\simeq  E\otimes f_{X*}1_{D(X\times I)}
		\simeq E\otimes 1_{D(X)}
		\simeq E.
		$$
		Thus the functor $f_X^*:D(X)\to D(X\times I)$
		is fully faithful.
		
	\end{proof}

	\begin{lem} \label{open and closed fiber seq}
Let $D\in 6\mathrm{FF}(\CondAn^\light,E)^\cont$.
		Let $X$ be a locally compact Hausdorff space,
		$i:Z\hookrightarrow X$ a closed subset and
		$j:U=X\setminus Z \hookrightarrow X$ the complement open subset. We have the fiber sequence
		$$
		j_! 1_{D(U)} \to 1_{D(X)} \to i_* 1_{D(Z)}
		$$
		in $D(X)$.
	\end{lem}
\begin{proof}
	By Corollary \ref{preserves delta_f and D_f},
	since $j$ is $\Shv(-;\Sp)$-\'etale 
	and $i$ is $\Shv(-;\Sp)$-proper,
	we know that
$j$ is $D$-\'etale and $i$ is $D$-proper,
which by \cite[4.6.4]{HM24} shows that
 $i_*\simeq i_!$ and $j^*\simeq j^!.$
	By proper base change, 
	since $U\times_X Z=0$, 
	we have
	$j^* i_* \simeq 0$,
	hence
	$ j^*\fib( 1_{D(X)} \to i_* 1_{D(Z)} )
	\simeq 1_{D(U)} $,
	which gives the map
	$$u: j_! 1_{D(U)}\to \fib( 1_{D(X)} \to i_* 1_{D(Z)} ) $$
	via adjunction.
	Since $D$ satisfies condition (a) in Theorem \ref{initial},
	the functors $i^*$ and $j^*$ are jointly conservative.
	It is clear that $i^* u$ and $j^* u$ are equivalences.
		Thus the map 
		$ j_! 1_{D(U)}\to \fib( 1_{D(X)} \to i_* 1_{D(Z)} ) $
		is an equivalence.
\end{proof}

	\begin{lem} \label{f^! 1 and f_! 1}
	Let $D\in 6\mathrm{FF}(\CondAn^\light,E)^\cont$.
		If $f:\mathbb{R}\to *$
		is the projection,
		then 
		\begin{itemize}
			\item [(1)] 
			$f^! 1_{D(*)}\simeq  1_{D(\mathbb{R})}[1]$.
			\item [(2)]
			$f_! 1_{D(\mathbb{R})}\simeq 1_{D(*)}[-1]$.
		\end{itemize}
		
	\end{lem}
	\begin{proof}
		For (1), consider the functor
		$\alpha_\mathbb{R}: \Shv(\mathbb{R};\Sp)
		\to  D(\mathbb{R})$.
		Since $f:\mathbb{R}\to *$
		is $\Shv(-;\Sp)$-suave,
		by Corollary \ref{preserves delta_f and D_f},
		the functor $\alpha_\mathbb{R}$
		sends $f^!1_{\Shv(\mathbb{R};\Sp)}$
		to $f^! 1_{D(*)}$.
		By \cite[4.8.9]{HM24}, we have
		 $f^! 1_{\Shv(\mathbb{R};\Sp)}
		\simeq 1_{\Shv(\mathbb{R};\Sp)}[1]$.
		Since the functor $\alpha_\mathbb{R}$ preserves colimits, we get
		$$
		f^! 1_{D(*)}
		\simeq \alpha_{\mathbb{R}}( f^!1_{\Shv(\mathbb{R};\Sp)} )
		\simeq \alpha_{\mathbb{R}}(1_{\Shv(\mathbb{R};\Sp)}[1] )
		\simeq \alpha_{\mathbb{R}}(1_{\Shv(\mathbb{R};\Sp)})[1] 
		\simeq 1_{D(\mathbb{R})}[1].
		$$
		For (2), condider the diagram
		\begin{center}
			\begin{tikzcd}
				\mathbb{R}\simeq(0,1)\ar[r,hook,"j"]\ar[rd,"f"swap]
				& I \ar[d,"p"] & \{0,1\}\ar[ld]\ar[l,hook',"i"swap] \\
				& * .&
			\end{tikzcd}
		\end{center}
		By Proposition \ref{f_X^* fully faithful},
		we have $p_* 1_{D(I)}\simeq 1_{D(*)}$,
		thus we get
		\begin{align*}
			f_! 1_{D(\mathbb{R})}
			&\simeq p_! j_! 1_{D(\mathbb{R})}
			\simeq p_* j_! 1_{D(\mathbb{R})}
			\simeq \mathrm{fib}
			( p_* 1_{D(I)}\to p_* i_* 1_{ D(\{0,1\}) } )\\
			&\simeq
			\fib( 1_{D(*)}\to 1_{D(*)} \oplus 1_{D(*)}  )
			\simeq 1_{D(*)}[-1],
		\end{align*}
	where the third equivalence is by Lemma 
	\ref{open and closed fiber seq}.	
	\end{proof}

	\begin{thm}
	Let $D\in 6\mathrm{FF}(\CondAn^\light,E)^\cont$.
		For any light condensed anima $X$
		with projection $f_X:X\times \mathbb{R} \to X$,
		the pullback functor
		$ f_X^*:D(X)\to D(X\times \mathbb{R}) $
		is fully faithful.
	\end{thm}
	\begin{proof}
The proof here is the same as the proof in \cite[4.28]{Zhu25}.		
Consider the square
		\begin{center}
			\begin{tikzcd}
				X\times \mathbb{R}\ar[r,"p_\mathbb{R}"]\ar[d,"f_X"]
				& \mathbb{R}\ar[d,"f"]\\
				X \ar[r,"p"]
				& *.
			\end{tikzcd}
		\end{center}
		By Lemma \ref{f^! 1 and f_! 1},
		we have $f_! 1_{D(\mathbb{R})}\simeq 1_{D(*)}[-1]$,
		thus by base change, we have
		$$ f_{X!} 1_{D(X\times \mathbb{R})}
		\simeq f_{X!} p_\mathbb{R}^* 1_{D(\mathbb{R})}
		\simeq p^* f_! 1_{D(\mathbb{R})} 
		\simeq  p^* ( 1_{D(*)}[-1])
		\simeq 1_{D(X)}[-1] .$$
		The map $f:\mathbb{R}\to *$ is $\Shv(-;\Sp)$-suave,
		and hence is $D$-suave, we get
		$$
		f_X^! 1_{D(X)}
		\simeq f_X^! p^* 1_{D(*)}
		\simeq  p_\mathbb{R}^* f^! 1_{D(*)}
		\simeq p_\mathbb{R}^* (1_{D(\mathbb{R})}[1])
		\simeq 1_{D(X\times \mathbb{R})}[1],
		$$
		where the third equivalence is by Lemma \ref{f^! 1 and f_! 1}.
		Thus we have
		$$f_X^!(-) \simeq  f_X^! 1_{D(X)}\otimes f_X^*(-) 
		\simeq 1_{D(X\times \mathbb{R})}[1]\otimes f_X^*(-)
		\simeq f_X^*(-)[1].$$
		Now, for any $E\in D(X)$, we have
		\begin{align*}
			f_{X*} f_X^* E
			&\simeq  f_{X*}  f_X^! E [-1]
			\simeq f_{X*} \underline{\Hom}(1_{D(X\times \mathbb{R})},
			f_X^! E)[-1]\\
			&\simeq \underline{\Hom}(f_{X!}1_{D(X\times \mathbb{R})},E )[-1]
			\simeq \underline{\Hom}(1_{D(X)}[-1],E)[-1]
			\simeq E.
		\end{align*}
		Thus  $ f_X^*:D(X)\to D(X\times \mathbb{R}) $
		is fully faithful.
	\end{proof}

	\small
	\bibliographystyle{amsalpha}
	\bibliography{universal} 
	
\end{document}